\documentclass[12pt, oneside,reqno]{amsart}
\pdfoutput=1

\usepackage[a4paper,margin=1in]{geometry}
\usepackage{amsmath,amssymb,amsthm,mathtools}
\usepackage{graphicx}
\usepackage{subcaption}
\usepackage{float}
\usepackage{tikz}
\usepackage{xcolor}
\usepackage{bbold}

\usepackage[normalem]{ulem}

\usepackage{chngcntr}

\usepackage[colorlinks=true,linkcolor=blue,citecolor=blue,urlcolor=blue]{hyperref}
\usepackage[nameinlink,noabbrev]{cleveref}
\numberwithin{figure}{subsection}

\newcommand{\C}{\mathbb{C}}
 
\newcommand{\R}{\mathbb{R}}
\newcommand{\Z}{\mathbb{Z}}

\newcommand{\Q}{\mathbb{Q}}

\theoremstyle{plain}
\newtheorem{thm}{Theorem}[section]     
\newtheorem{lemma}[thm]{Lemma}

\newtheorem{corollary}[thm]{Corollary}

\theoremstyle{definition}
\newtheorem{defn}[thm]{Definition}
\newtheorem{remark}[thm]{Remark}

\author{Ghaith Hiary$^{*}$, Ali Saraeb$^{\ddagger}$}

\address{$^{*}$Mathematics Department, The Ohio State University, 231 W. 18th Avenue,  Columbus, Ohio 43210, USA}
\address{$^{\ddagger}$Mathematics Department, The Ohio State University, 231 W. 18th Avenue,  Columbus, Ohio 43210, USA}

\email{hiary.1@math.osu.edu}
\email{saraeb.1@osu.edu}

\title{Functional Equations Characterize Dirichlet Characters}

\begin{document}

\begin{abstract}
    We prove a converse theorem for functional equations of Dirichlet $L$-functions. Under mild assumptions, we prove that these functional equations for $L$-series of the form $\sum_{n\ge 1} f(n) n^{-s}$
    force the coefficient function $f$ to be a primitive Dirichlet character. Consequently, these functional equations force the existence of an Euler product.
\end{abstract}


\maketitle

\section{Introduction}
Primitive Dirichlet characters have both algebraic and analytic structures. Algebraically, they are multiplicative functions on residue classes, and analytically, their completed $L$-functions satisfy functional equations. It is classical that the algebraic structure gives rise to the analytic one. In this paper, we ask the converse question: can such functional equations force the coefficient functions to be primitive characters?

This question fits into a broader theme, where one asks about properties that characterize Dirichlet characters. Most existing characterizations assume complete multiplicativity and other hypotheses, and then prove that the coefficient function is a Dirichlet character \cite{allouche2018mock,glazkov1968characters,konieczny2025multiplicative,klurman2018rigidity}. The question we are concerned with, however, is of a different flavor: can a non-multiplicative hypothesis force multiplicativity? A question of this kind was originally posed by Cohn \cite[p. 202]{montgomery1994ten}, who asked: 

\begin{quote}
    \itshape
    If $F$ is a finite field, $f: F \to \C,$ $f(0)=0,$ $|f(a)|=1$ for all $a \neq 0,$ $f(1)=1,$ and 
    \[
    \sum_{b \in F} f(b) \overline{f(a+b)} =-1,
    \]
    for all $a \neq 0,$ does it follow that $f$ is a character of $F$?
\end{quote}
This question was answered in the affirmative for finite fields of prime order by Kurlberg \cite{kurlberg2002character} and more recently by Bober and Goldmakher \cite{bober2024converse}. However, Choi and Siu then proved the answer is negative for non-prime finite fields  \cite{choi2000counter}. Our results give a new non-multiplicative criterion for primitive Dirichlet characters of $\Z/ N\Z.$

Let $N>1,$ and let $\chi$ be a primitive Dirichlet character modulo $N.$ Let $a \in \{0,1\}$ be the parity of $\chi$, defined by $\chi(-1)=(-1)^a.$ The associated Dirichlet $L$-function is 
\begin{equation}
    L(s,\chi):= \sum_{n\ge 1}  \frac{\chi(n)}{n^s}, \qquad \Re(s)>1.
\end{equation}

Such an arithmetic function $\chi$ gives rise to an entire completed Dirichlet $L$-function
\begin{equation} \label{eq: Lambda def}
    \Lambda(s,\chi):=\Bigl(\frac{N}{\pi}\Bigr)^{\frac{s+a}{2}} \Gamma \Bigl(\frac{s+a}{2}\Bigr) L(s,\chi),
\end{equation}
which satisfies the functional equation \cite[eq. (13)-(14), Chp. 9]{davenport2013multiplicative}
\begin{equation} \label{eq: funct eq}
     \Lambda(s,\chi)=\varepsilon_\chi \Lambda(1-s,\overline\chi).
\end{equation}
Here $\varepsilon_\chi$ is the root number of $\chi$ given by
\begin{equation}
    \varepsilon_\chi= \frac{\tau(\chi)}{i^a\sqrt N}, \qquad |\varepsilon_\chi|=1,
\end{equation}
where
\begin{equation}
    \tau(\chi) :=\sum_{x=1}^N \chi(x) e^{ \frac{2 \pi i x}{N}}
\end{equation}
is the Gauss sum of $\chi$.

The goal of this paper is to prove a converse statement. Namely, we show, under mild assumptions, that if the associated $L$-series satisfy functional equations of the form \eqref{eq: funct eq}, then the coefficient function is necessarily a primitive Dirichlet character. In this sense, functional equations of the form \eqref{eq: funct eq} characterize primitive Dirichlet characters. 

The conductor and order of a primitive Dirichlet character are not independent. For a primitive character $\chi$ of conductor $N$ and order $m,$ we show that 
\begin{equation}
    \begin{cases}
        2^\alpha\parallel N, \alpha \ge 2   \implies 2 \mid m,   \\
        p^\alpha \parallel N, \alpha \ge 2 \implies p \mid m, \qquad p \text{ odd}.
    \end{cases}
\end{equation}
In particular, if $(m,N)=1,$ then $N$ is necessarily odd and squarefree. It thus appears that the case $(m,N)=1$ is structurally simpler than the case $(m,N)>1$. We prove that when $(m,N)=1,$ one functional equation of the form \eqref{eq: funct eq} is enough to force both multiplicativity and primitivity. When $(m,N)>1,$ however, this is no longer true in general (see Section \ref{sec: conv to Gauss thm}). Our theorem gives a sufficient condition: it is enough to impose such functional equations for a set of powers of the arithmetic function. 

We need the following terminology to state the main theorem. For $N>1$ and $f: \Z/N\Z \to \C$, we extend $f$ periodically to an arithmetic function on $\Z,$ and we say $f$ has parity $a \in \{0,1\}$ if 
\begin{equation}
    f(-x)= (-1)^a f(x), \qquad x \in  \Z.
\end{equation}

\begin{thm} \label{thm: mainthm}
    Let $N,m>1$ be integers. Let $f: \Z /N\Z \to \mu_m \cup\{0\}$ have order $m$ and parity $a \in \{0,1\},$ with $f(1)=1$ and $f(n)=0$ iff $(n,N)>1.$ Put $d=(m,N),$ and let $H\leq(\Z/m\Z)^\times$ be such that the reduction map $\pi : H \to (\Z/d\Z)^\times$ is surjective. Assume that, for each $c \in H,$ $L(s, f^c)$ is holomorphic at $s=1$ and that there exists $\varepsilon_c \in\C^\times$ such that 
    \begin{equation}\label{eq: func eq lambda fc}
        \Lambda(s,f^c)=\varepsilon_c \Lambda(1-s,\overline {f^c}), \qquad s \in \C.
    \end{equation}
    Then $f$ is a primitive character modulo $N.$ Furthermore, if $(m,N)=1,$ then $N$ is necessarily odd and squarefree.
\end{thm}

\begin{remark} \label{rem: on H}
    When $(m,N)=1$, the theorem applies with $H= \{1\},$ and when $(m,N)>1$, the theorem always applies with the choice $H=(\Z/m\Z)^\times,$ since the reduction map $\pi$ is surjective for $d \mid m.$ Assuming \eqref{eq: func eq lambda fc} only for $f$ is insufficient in general (see Section \ref{sec: conv to Gauss thm}).
\end{remark}

\begin{remark}
    If $f$ is a primitive character modulo $N$ of order $m$ and $c \in (\Z /m\Z)^\times,$ then $f^c$ is also a primitive character modulo $N$ of order $m$ (see Section \ref{sec: conv to Gauss thm}).
\end{remark}

In particular, we have the following immediate consequence.

\begin{corollary}
    Under the hypotheses of \Cref{thm: mainthm}, $L(s,f)$ admits the Euler product 
    \begin{equation}
        L(s,f)= \prod_p \Bigl( 1- \frac{f(p)}{p^{s}} \Bigr)^{-1}, \qquad \Re(s)>1.
    \end{equation}
\end{corollary}

The main ingredient in the proof of \Cref{thm: mainthm} is a converse to Gauss's theorem on the separability of Gauss sums of primitive Dirichlet characters, and we expect this result to be of independent interest. To state the result, for a function $f: (\Z /N\Z)^\times \to \C^\times$, define the normalized Fourier transform by
\begin{equation}
    \hat f(\xi):= \frac{1}{\sqrt{N}} \sum_{x \in \Z /N\Z} f(x) e(-\frac{ x \xi}{N}), \qquad \xi  \in \Z/N\Z.
\end{equation}

\begin{thm} \label{thm: prop1 case 2}
    Let $N,m>1$ be integers. Let $f: \Z /N\Z \to \mu_m \cup\{0\}$ have order $m,$ with $f(1)=1$ and $f(n)=0$ iff $(n,N)>1.$ Put $d=(m,N),$ and let $H\leq(\Z/m\Z)^\times$ be such that the reduction map $\pi : H \to (\Z/d\Z)^\times$ is surjective. Assume that for each $c\in H,$ there exists $\lambda_c \in \C^\times$ such that for all $\xi \in \Z /N\Z,$
    \begin{equation} \label{eq: assume prop1 case 2}
        \widehat{f^c}(\xi)= \lambda_c \overline{f^c(\xi)}, \qquad \xi \in \Z /N\Z.
    \end{equation}
    Then $f$ is a primitive character modulo $N.$ Furthermore, if $(m,N)=1,$ then $N$ is necessarily odd and squarefree.
\end{thm}

\begin{remark}
    The statement of Remark \ref{rem: on H} holds for \Cref{thm: prop1 case 2}.
\end{remark}

\section{A converse theorem for Gauss sums} \label{sec: conv to Gauss thm}

Our first and main step in proving \Cref{thm: mainthm} is to prove a converse to Gauss's theorem on the separability of Gauss sums of primitive Dirichlet characters. We shall show in Section \ref{sec: the main results} that \Cref{thm: mainthm} is a consequence of this result.

Let $N >1$ be an integer. We seek natural sufficient conditions under which a function $f: (\Z /N\Z)^\times \to \C^\times$ is necessarily a primitive Dirichlet character modulo $N$. Throughout, we assume that $f$ has a finite order $m\ge 1$, i.e.
\begin{equation} \label{eq: img f}
    \text{im}(f):= \{ f(x): x \in (\Z /N\Z)^\times \} \subset \mu_m,
\end{equation}
where $\mu_m$ denotes the set of $m^{th}$ roots of unity and $m$ is chosen to be the minimal natural number satisfying \eqref{eq: img f}.

In the study of Dirichlet characters, it is standard to consider their inner product with the additive character $e(\cdot/N)$ \cite[Chp. 9]{Montgomery_Vaughan_2006}, where we write
\begin{equation}
    e(x):= e^{2\pi i x}.
\end{equation}
Extending $f$ to $\Z/N\Z$ by setting $f(x)=0$ whenever $(x,N)>1$, we study the Gauss sum of $f$ defined as
\begin{equation}
    \tau(f,n):= \sum_{x \in \Z/N\Z} f(x) e(\frac{nx}{N}), \qquad n \in \Z/N\Z.
\end{equation}
The Gauss sum can be viewed as a discrete Fourier transform of $f$ evaluated at $-n.$ When $n=1,$ we shall write $\tau(f)$ for $\tau(f,1).$

For a primitive Dirichlet character $f$ modulo $N$, the Gauss sum satisfies the classical identities
\begin{equation} \label{eq: prop1}
    f(n) \tau \Bigl( \bar{f} \Bigr)= \tau \Bigl( \bar{f}, n \Bigr), \qquad  n \in \Z/N\Z,
\end{equation}
and
\begin{equation} \label{eq: prop2}
    |\tau(f)| = \sqrt{N}.
\end{equation}
See, for example, \cite[Thm. 9.7]{Montgomery_Vaughan_2006}. We shall refer to identity \eqref{eq: prop1} as the separability of Gauss sums.

\begin{remark}
    If $f$ is a Dirichlet character, then $|f(n)|=1$ for every $(n,N)=1.$ Hence the above two identities \eqref{eq: prop1} and \eqref{eq: prop2} imply
    \begin{equation}
        \Bigl|\tau \Bigl( \bar{f}, n \Bigr)\Bigr|=\sqrt{N}, \qquad (n,N)=1.
    \end{equation}
\end{remark} 
It is natural to ask whether the converse holds. In other words, if $f$ satisfies an analogue of \eqref{eq: prop1} or \eqref{eq: prop2}, does it follow that $f$ is a primitive Dirichlet character? Questions of this type have attracted attention in the past two decades; see, for example, \cite{kurlberg2002character, bober2024converse, choi2000counter}.

In a recent paper \cite{bober2024converse}, Bober and Goldmakher partially answered the question in the affirmative, under some mild assumptions, for the case of $N=p$ prime. Specifically, they assumed that $f(0)=0,$ $f(1)=1,$ and the image of $f$ consists of $m^{th}$ roots of unity, with $p \nmid m,$ and they proved that if \eqref{eq: prop2} holds, then $f$ is a primitive character \cite[Thm. 1.2]{bober2024converse}. In this section, we complement \cite{bober2024converse, kurlberg2002character} by presenting a converse theorem to the separability property \eqref{eq: prop1} for general moduli $N.$ This is the form needed in the present paper.

We emphasize, however, that the identity \eqref{eq: prop2} is much weaker than the separability identity \eqref{eq: prop1}. In fact, the analogue of the result proven in this section, with the assumption \eqref{eq: prop1} replaced by \eqref{eq: prop2}, is false in general (see Remark \ref{note: prop2 false}). A converse to \eqref{eq: prop2} for general moduli $N$ requires different algebraic methods and assumptions, and will be treated in a forthcoming paper.

One first asks whether the separability of $f$
\begin{equation}\label{eq: f sep}
    \widehat{f}(\xi)= \lambda \overline{f(\xi)}, \qquad \xi \in \Z /N\Z,
\end{equation}
for some $\lambda \in \C^\times$, is sufficient to force the multiplicativity of $f.$ This is true when $(m,N)=1$, and the proof, in this case, uses the essential fact that, for every $s\in(\Z/N\Z)^\times,$ it is possible to construct an automorphism of $\Q(\zeta_{mN})$ that fixes $\Q(\zeta_m)$ and maps $\zeta_N$ to $\zeta_N^s.$ 

This is no longer possible when $(m,N)>1.$ It turns out that this is not a limitation of the proof, since functions satisfying the Fourier identity \eqref{eq: f sep} need not be multiplicative in general. Indeed, consider the following counterexample. Let
\begin{equation}
    f:\Z/9\Z\to \mu_3\cup\{0\}
\end{equation}
be defined by
\begin{equation} \label{eq: f counterex}
f(n)=
\begin{cases}
0, & 3\mid n, \\
1, & n\equiv 1,2 \pmod 9,  \\
e(2/3), & n\equiv 4,8 \pmod 9, \\
e(1/3), & n\equiv 5,7 \pmod 9.
\end{cases}
\end{equation}
Then $f(1)=1$ and $f(n)=0$ iff $(n,9)>1.$ Moreover, 
\begin{equation} \label{eq: f counterex hat f}
    \hat f(\xi)=e(7/9) \overline{f(\xi)},
\qquad \xi\in \Z/9\Z.
\end{equation}
However, $f$ is clearly not multiplicative, since for example, 
\begin{equation}\label{eq: f not mul}
    f(2)^2 = 1 \neq e(2/3) = f(4).
\end{equation}
On the contrary, there is a primitive Dirichlet character modulo $9$ of order $3$. Namely, 
\begin{equation}
    \chi_9(n): =
    \begin{cases}
        1, & n\equiv 1,8 \pmod 9, \\
        e(1/3), & n\equiv 2,7 \pmod 9, \\
        e(2/3), & n\equiv 4,5 \pmod 9.
    \end{cases}
\end{equation}
In other words, there are multiplicative functions $f$ satisfying the Fourier identity \eqref{eq: f sep} in the case of $N=9,$ $m=3.$ Therefore, the separability identity \eqref{eq: prop1} is generally insufficient by itself to imply multiplicativity in the case where $(m, N)>1$.

To formulate the correct converse, we recall the following property of primitive Dirichlet characters. Suppose $\chi: (\Z /N\Z)^\times \to \mu_m $ is a primitive character modulo $N$ of order $m.$ Then for every $c \in (\Z /m\Z)^\times,$ the character 
\begin{equation}
    \chi^{c}(x):= \chi(x)^c, \qquad  x \in \Z /N\Z,
\end{equation}
is again primitive modulo $N$ of order $m.$ 

To see why this is true, recall that a character $\psi$ modulo $N$ is primitive iff it has no induced modulus $d < N$ \cite[p. 221]{apostol2013introduction}. An induced modulus $d$ for $\psi$ is a divisor of $N$ such that 
\begin{equation}
    \psi(n)= 1,
\end{equation}
whenever
\begin{equation} \label{eq: ind mod n}
    (n,N)=1, \qquad n \equiv 1 \pmod d.
\end{equation}
Now let $\chi$ be a primitive character modulo $N$ of order $m,$ and suppose, for a contradiction, that there exists $c \in (\Z /m\Z)^\times,$ such that $\chi^c$ is imprimitive modulo $N.$ Then there exists a divisor $d$ of $N$ such that 
\begin{equation} \label{eq: chi c }
    \chi(n)^c=1
\end{equation}
for all $n$ satisfying \eqref{eq: ind mod n}. However, $c$ has an inverse $c^{-1} \in (\Z /m\Z)^\times.$ Thus, since $\chi(n) \in \mu_m$ for all $(n,N)=1,$ it follows by raising both sides of \eqref{eq: chi c } to the power $c^{-1}$ that $\chi$ has an induced modulus $d$, contradicting primitivity of $\chi$.

Thus, in particular, if $\chi$ is such a primitive character, then  $\chi^{c}$ satisfies the separability property \eqref{eq: prop1} for every $c \in (\Z /m\Z)^\times.$ It turns out that this property characterizes primitive characters in the sense of \Cref{thm: prop1 case 2}. 

\begin{remark} \label{note: prop2 false}
    We note that the analogue of \Cref{thm: prop1 case 2}, with the separability assumption \eqref{eq: prop1} replaced by \eqref{eq: prop2}, is false in general.
    Indeed, consider the example in \eqref{eq: f counterex}, where $N=9$ and $m=3$. By \eqref{eq: f counterex hat f}, it is clear that $|\tau(f)| = \sqrt{N}.$ Moreover, from the definition \eqref{eq: f counterex}, we see that $f^2= \bar{f},$ so that conjugating \eqref{eq: f counterex hat f} implies that $|\tau(f^2)| = \sqrt{N}.$ However, by \eqref{eq: f not mul}, $f$ is not multiplicative.
\end{remark}

To prove \Cref{thm: prop1 case 2}, it is convenient to introduce the unnormalized Fourier transform of $f$, defined by
\begin{equation}\label{def: FN f}
    \mathcal F f(\xi):=  \sum_{x \in \Z /N\Z} f(x) e(-\frac{ x \xi}{N})= \sqrt{N} \hat f(\xi), \qquad \xi  \in \Z/N\Z.
\end{equation}

\begin{proof}[proof of \Cref{thm: prop1 case 2}]
    Throughout, we write $\zeta_q: = e(1/q)$ for $q \in \Z_{\ge1}.$ We begin by rewriting \eqref{eq: assume prop1 case 2} in terms of \eqref{def: FN f}. For every $c\in H,$ write 
    \begin{equation} \label{eq: Ffc}
        \mathcal{F} [f^c](\xi)= \Lambda_c \overline{f^c(\xi)}, \qquad \xi \in \Z /N\Z,
    \end{equation}
    where $\Lambda_c := \sqrt{N} \lambda_c$. In particular, $\Lambda_c \neq0$ for every $c\in H$. Moreover, evaluating \eqref{eq: Ffc} at $c=1$ and $\xi=1$ and using $f(1)=1,$ we observe that 
    \begin{equation} \label{eq: Lambda FN f}
         \Lambda_1= \mathcal F f(1)=  \sum_{x \in \Z /N\Z} f(x) \zeta_N^{- x} \in \Q(\zeta_{mN}).
    \end{equation}

    We first prove $f$ is multiplicative. Let $r,s \in (\Z /N\Z)^\times.$ Since $(s,N)=1$ and $d \mid N,$ we have $(s,d)=1.$ Thus, since $\pi : H \to (\Z/d\Z)^\times$ is surjective, it follows that there exists $c_s \in H$ such that $c_s \equiv s \pmod d.$

    Therefore, by the Chinese remainder theorem, there exists an integer $x_s$ such that 
    \begin{equation} \label{eq: CRT}
        x_s \equiv s \pmod N, \qquad x_s \equiv c_s \pmod m.
    \end{equation}
    
    Using \eqref{eq: CRT}, $(s,N)=1,$ and $c_s \in H$, we have $(x_s, mN)=1.$ Therefore, the map 
    \begin{equation} \label{eq: auto case1}
        \sigma_s: \Q(\zeta_{mN}) \to \Q(\zeta_{mN}), \qquad \sigma_s(\zeta_{mN})= \zeta_{mN}^{x_s},
    \end{equation}
    is a field automorphism of $\Q(\zeta_{mN}),$ with the property  
    \begin{equation}
        \sigma_s( \zeta_N)= \zeta_N^s, \qquad \sigma_s( \zeta_m)= \zeta_m^{c_s}.
    \end{equation}
    In particular, since $f(x)$ is an $m^{th}$ root of unity for every $x \in (\Z /N\Z)^\times$ and is $0$ otherwise, we have 
    \begin{equation}
        \sigma_s \bigl(f(x) \bigr)= f(x)^{c_s}, \qquad x \in \Z /N\Z.
    \end{equation}
    Therefore, applying $\sigma_s$ to \eqref{eq: Ffc} with $c=1$ and $\xi=r$,  we obtain 
    \begin{equation} \label{eq: F fu}
        \mathcal{F} [f^{c_s}](sr)= \sigma_s \bigl(  \Lambda_1 \bigr) f(r)^{-c_s}.
    \end{equation}
    On the other hand, using \eqref{eq: Ffc} with $c=c_s$ and $\xi=sr,$ and substituting into \eqref{eq: F fu}, we get 
    \begin{equation} \label{eq: F fu2}
       \Lambda_{c_s} f(sr)^{-c_s}= \sigma_{s} \bigl(  \Lambda_1 \bigr) f(r)^{-c_s}.
    \end{equation}
    Hence \eqref{eq: F fu2} holds for every $s, r \in (\Z /N\Z)^\times$. In particular, evaluating \eqref{eq: F fu2} at $r=1$ and using $f(1)=1,$ it follows that 
    \begin{equation} \label{eq: sigm u lam}
    \Lambda_{c_s} f(s)^{-c_s}= \sigma_s \bigl(  \Lambda_1 \bigr).
    \end{equation}
    Therefore, substituting \eqref{eq: sigm u lam} into the right hand side of \eqref{eq: F fu2}, dividing by $\Lambda_{c_s},$ and conjugating both sides, \eqref{eq: F fu2} becomes
    \begin{equation} \label{eq: f u mul}
         f(sr)^{c_s}=  f(s)^{c_s} f(r)^{c_s}, \qquad s, r \in (\Z /N\Z)^\times.
    \end{equation}   
    Since $c_s\in(\Z/m\Z)^\times$ and $f(s), f(r), f(sr) \in \mu_m$, raising both sides of \eqref{eq: f u mul} to the inverse of $c_s$ modulo $m$ gives
    \begin{equation} \label{eq: f u mul2}
         f(sr)=  f(s) f(r), \qquad s, r \in (\Z /N\Z)^\times.
    \end{equation} 
    This proves multiplicativity, and hence $f$ is a Dirichlet character modulo $N$ since it vanishes exactly on the non-units. Primitivity is then immediate by Gauss characterization of primitive characters \cite[Thm. 8.19]{apostol2013introduction}. Indeed, using \eqref{eq: Ffc} with $c=1$, \eqref{eq: Lambda FN f}, and multiplicativity of $f$, we obtain
    \begin{align}
        \tau ( f, \xi ) &= \mathcal F f(-\xi) = \Lambda_1 \overline{f(-\xi)}=\Lambda_1 \overline{f(-1)} ~ \overline{f(\xi)} \nonumber \\
        &= \mathcal F f(-1) \overline{f(\xi)}= \tau ( f) \overline{f(\xi)}, \qquad \xi  \in \Z/N\Z.
    \end{align}
    In other words, the Gauss sum $\tau ( f, \xi )$ is separable, whence $f$ is primitive. For more details, see \cite[p. 165-173]{apostol2013introduction}.
    
    Lastly, suppose that $(m,N)=1.$ We show that $N$ must be odd and squarefree. We first rule out the possibility that $N$ is even. Suppose, for a contradiction, that $N=2^\nu q $ with $\nu \ge 1$ and $q$ odd. If $\nu=1,$ then $f$ cannot be primitive, since there are no primitive characters modulo $2,$ a contradiction. If $\nu \ge 2,$ then $f= \chi_{2^\nu} \chi_q$ where $\chi_{2^\nu}, \chi_q $ are primitive characters modulo $2^\nu, q,$ respectively. We now derive a contradiction in the cases $\nu \ge 3$ and $\nu=2,$ separately.
    
    First, if $\nu \ge 3,$ then there exists an odd integer $1 \le c < 2^{\nu-2}$ such that for every odd integer $n$, there is a unique integer $1 \le b(n) < 2^{\nu-2}$ for which
    \begin{equation} \label{eq: char mod 2nu}
        \chi_{2^\nu}(n)= \pm e \Bigl( \frac{c~ b(n)}{2^{\nu-2}} \Bigr), \qquad 2 \nmid n.
    \end{equation}
    See, for example, \cite[p. 218, 219]{apostol2013introduction}.  Moreover, the oddness of $c$ follows from primitivity of $\chi_{2^\nu}$ and \cite[Thm. 10.15]{apostol2013introduction}.
    
    In particular, by \cite[Thm. 10.11]{apostol2013introduction}, we see that $b(5)=1,$ and thus $\text{ord}\bigl(\chi_{2^\nu}(5)\bigr)= 2^{\nu-2}$ is even\footnote{For $\nu=3,$ primitive characters modulo $8$ satisfy $\chi_{8}(5)=-1,$ which is of order $2$.}. Now by the Chinese remainder theorem, we can choose $u\in(\Z/N\Z)^\times$ such that
    \begin{equation}
            u\equiv 5 \pmod {2^\nu}, \qquad u\equiv 1 \pmod q,
    \end{equation}
    which implies that $f(u)= \chi_{2^\nu}(5)$ has an even order. This is impossible since this implies that $2 \mid \text{ord}\bigl(f(u)\bigr) |m$ (recall $f(u) \in \mu_m$), contradicting  $(m,N)=1.$ 
    
    Similarly, we can rule out the case $\nu=2,$ since the unique primitive character modulo $4$ satisfies $\chi_{2^\nu}(3)=-1,$ which is of even order, thereby again contradicting $(m,N)=1.$ 
    
    This proves that $N$ is necessarily odd. Now suppose, for a contradiction, that $N$ were not squarefree. Then there exist an odd prime $p$ and an integer $\alpha>1$ such that $p^\alpha \parallel N.$ 
    
    Write $N= p^\alpha M$ with $(p,M)=1.$ Since $f$ is a primitive character modulo $N,$ $f$ decomposes as $f= \chi_{p^\alpha} \chi_M,$ where $\chi_{p^\alpha}, $ $\chi_M$ are primitive characters modulo $p^\alpha$, $M,$ respectively. 
    
    Let $g$ be a primitive root modulo $p^\alpha.$ Then the primitive character $\chi_{p^\alpha}$ is given by
    \begin{equation}
        \chi_{p^\alpha}(n) = e \Bigl( \frac{k~ ind_g ~n}{\phi(p^\alpha)} \Bigr), \qquad  p \nmid n
    \end{equation}
    for some integer $k$ with $p \nmid k$ (see, e.g. \cite[p. 284]{Montgomery_Vaughan_2006}, \cite[Thm. 10.14]{apostol2013introduction}). In particular, 
    \begin{equation}
        \text{ord}\bigl(\chi_{p^\alpha}(g)\bigr)= \frac{\phi(p^\alpha)}{\bigl(k,\phi(p^\alpha) \bigr)}.
    \end{equation}
    Thus, we clearly have that $p \mid \text{ord}\bigl(\chi_{p^\alpha}(g)\bigr).$ 
    
    However, by the Chinese remainder theorem, there exists $h \in (\Z/N\Z)^\times$ such that 
    \begin{equation}
        h \equiv g \pmod {p^\alpha}, \qquad h \equiv 1 \pmod M,
    \end{equation}
    whereby $f(h)= \chi_{p^\alpha}(g)  \chi_M(1)= \chi_{p^\alpha}(g).$ Thus, $p \mid \text{ord}\bigl(f(h)\bigr)\mid m$ (recall $f(h) \in \mu_m$), contradicting $(m,N)=1.$ This completes the proof. 
\end{proof}

\section{Proof of \Cref{thm: mainthm}} \label{sec: the main results}

In this section, we prove the main theorem \Cref{thm: mainthm} using the converse theorem, \Cref{thm: prop1 case 2}, for Gauss sums that was established in the previous section.

Throughout, let $N>1$ and $f: \Z/N\Z \to \C$ be given, and extend $f$ periodically to an arithmetic function on $\Z.$ 

A classical proof of the functional equation for a Dirichlet $L$-function attached to a primitive Dirichlet character proceeds through theta functions. Depending on the parity of the character, one first constructs a corresponding theta function whose Mellin transform equals the completed $L$-function. Poisson summation is then used to derive a transformation law for this theta function, and this transformation law yields the functional equation of the completed $L$-function. Therefore, we will first need to understand some properties of such theta functions, their connection to Gauss sums, and the consequences of their transformation laws.

\begin{defn}
    We define the twisted theta function associated to $f$ by 
    \begin{equation}
        \psi(t,f) := \sum_{n\in\Z}f(n)n^a e^{-\pi n^2t/N}, \qquad t>0,
    \end{equation}
    where we interpret $n^a \equiv 1$ if $a=0$.
\end{defn}

If $f$ is a primitive Dirichlet character modulo $N,$ then $\psi(t,f)$ satisfies the transformation law
\begin{equation} \label{eq: psi func eq prim}
    \psi(t,f)= \frac{\tau(f)}{i^a \sqrt{N}} t^{-a-1/2} \psi(t^{-1}, \bar{f}).
\end{equation}

For a general periodic arithmetic function $f,$ we have the following more general identity, expressing $\psi(t,f)$ in terms of the corresponding Gauss sums $\tau(f,n).$

\begin{lemma} \label{lem: func eq of psi}
    Let $N\ge 1,$ $a \in \{0,1\},$ and $f: \Z/N\Z \to \C$ be given. Extend $f$ periodically to $\Z.$ Then we have 
    \begin{equation} \label{eq: psi tau}
        \psi(t,f)= \frac{i^{-a}}{\sqrt N}t^{-a-\frac12} \sum_{n\in\Z} \tau(f,n)n^a e^{-\pi n^2/(Nt)}, \qquad t>0.
    \end{equation}
\end{lemma}
\begin{remark}
    Note that Lemma \ref{lem: func eq of psi} does not require $f$ to be a Dirichlet character.
\end{remark}
\begin{proof}[proof of Lemma \ref{lem: func eq of psi}.]
    Let $g(x):= x^a e^{- \pi x^2 t/N}.$ Since $a \in \{0,1\},$ then 
    \begin{equation} \label{eq: fourier of g}
        \hat{g}(\xi):= \int_{-\infty}^{\infty} g(x) e(- x \xi) dx = i^{-a} \Bigl( \frac{N}{t} \Bigr)^{a+1/2} \xi^a  e^{- \pi N \xi^2/t}, \qquad \xi \in \R.
    \end{equation}
    On the other hand, by absolute convergence, the periodicity of $f,$ and Poisson summation \cite[Thm 3.1, Chp. 5]{stein2011fourier}, we have 
    \begin{align}
        \psi(t,f) &= \sum_{r=0}^{N-1}   \sum_{m \equiv r \text{ mod } N}   f(m) g(m) = \sum_{r=0}^{N-1} f(r)   \sum_{m \equiv r \text{ mod } N}  g(m) \nonumber \\
        &= \sum_{r=0}^{N-1} f(r)   \sum_{k \in \Z}  g(r + k N)  =  \frac{1}{N} \sum_{r=0}^{N-1} f(r)   \sum_{n \in \Z} e \Bigl(\frac{nr}{N} \Bigr) \hat{g} \Bigl(\frac{n}{N} \Bigr) \nonumber \\
        &= \frac{1}{N}\sum_{n \in \Z}  \Bigl[ \sum_{r=0}^{N-1} f(r) e\Bigl(\frac{nr}{N} \Bigr) \Bigr] \hat{g} \Bigl(\frac{n}{N} \Bigr)= \frac{1}{N} \sum_{n \in \Z} \tau(f,n)  \hat{g} \Bigl(\frac{n}{N} \Bigr).
    \end{align}
    Using the Fourier transform of $g$ \eqref{eq: fourier of g}, the result follows. 
\end{proof}

We are now ready to establish the converse statements to the transformation law \eqref{eq: psi func eq prim}, from which our main theorem follows.

\begin{thm} \label{thm: theta converse case 2}
    Let $N,m>1$ be integers. Let $f:\Z/N\Z\to \mu_m\cup\{0\}$ have order $m$ and parity $a\in\{0,1\}$, with $f(1)=1$ and $f(n)=0$ iff $(n,N)>1$. Put $d=(m,N),$ and let $H\leq(\Z/m\Z)^\times$ be such that the reduction map $\pi : H \to (\Z/d\Z)^\times$ is surjective. Assume that, for each $c\in H$, there exists $\varepsilon_c\in\C^\times$ such that
    \begin{equation} \label{eq: assumed theta functional eq case 2}
        \psi(t,f^c)=\varepsilon_c \, t^{-a-1/2}\psi(t^{-1},\overline{f^c}),
        \qquad t>0.
    \end{equation}
    Then $f$ is a primitive character modulo $N$. Furthermore, if $(m,N)=1,$ then $N$ is necessarily odd and squarefree.
\end{thm}

\begin{remark}
    We again note that the theorem always applies with the choice $H=(\Z/m\Z)^\times$.
\end{remark}

\begin{remark}
    If $(m,N)=1$, then the theorem applies with $H= \{1\},$ and hence it suffices to assume one theta functional equation. 
\end{remark}

\begin{proof}[proof of \Cref{thm: theta converse case 2}]
We first note that, for every $c \in H,$ $f^c$ also has parity $a.$ Indeed, 
\begin{equation}
    f^c(-x)= f(-x)^c= (-1)^{ac} f^c(x), \qquad x \in \Z.
\end{equation}
If $a=0,$ then clearly $f^c$ has parity $a.$ If $a=1,$ then $f(-1)= -1,$ so $m$ is necessarily even. Thus, since $(c,m)=1,$ $c$ is odd, whence it follows that $f^c$ has parity $a.$

Substituting \eqref{eq: psi tau} for $f^c$ into \eqref{eq: assumed theta functional eq case 2}, cancelling the $t^{-a-1/2}$ factor, and putting $q=e^{-\pi/(Nt)},$ we obtain
\begin{equation} \label{eq: identity q 01}
    \frac{i^{-a}}{\sqrt N} \sum_{n\in\Z}\tau(f^c,n)n^a q^{n^2} =\varepsilon_c \sum_{n\in\Z}\overline{f^c(n)}n^a q^{n^2}, \qquad q \in (0,1).
\end{equation}
Both sides are clearly holomorphic power series in $q$ for $|q|<1,$ and hence \eqref{eq: identity q 01} holds on the unit disk $|q|<1$ by the identity theorem. 

On the other hand, since $f^c$ has parity $a,$ we have
\begin{equation} \label{eq: parity}
    \tau(f^c,-n)=(-1)^a\tau(f^c,n), \qquad \overline{f^c(-n)}=(-1)^a\overline{f^c(n)}.
\end{equation}
Therefore, comparing the coefficients of $q^{n^2}$ in \eqref{eq: identity q 01} for $n\ge 1,$ and using \eqref{eq: parity}, we obtain 
\begin{equation}
    \frac{2i^{-a}}{\sqrt N}\tau(f^c,n)n^a= 2\varepsilon_c \overline{f^c(n)}n^a,
\end{equation}
which implies 
\begin{equation}
    \tau(f^c,n)=i^a\varepsilon_c \sqrt N\,\overline{f^c(n)}, \qquad n \in \Z/N \Z.
\end{equation}

Consequently, since $\widehat{f^c}(\xi) = \frac{1}{\sqrt{N}} \tau(f^c,-\xi),$ it follows that 
\begin{equation}
    \widehat {f^c}(\xi)= \lambda_c \overline{f^c(\xi)}, \qquad \xi\in\Z/N\Z,
\end{equation}
with $\lambda_c:=(-1)^a i^a\varepsilon_c.$ Applying \Cref{thm: prop1 case 2} completes the proof. 
\end{proof}

It remains to pass from the functional equation of the completed $L$-function $\Lambda$ to the corresponding transformation law for the theta function $\psi$. This is a standard Mellin inversion argument, but we include the details below for completeness.

\begin{proof}[proof of \Cref{thm: mainthm}]
    For $\sigma \in \R,$ write $\int_{(\sigma)}$ for the integral over the vertical line $\Re(s)=\sigma.$ As in the proof of \Cref{thm: theta converse case 2}, $f^c$ also has parity $a$ for all $c \in H.$ 
    
    However, for a periodic function $g: \Z/N\Z \to \C$ with $g(0)=0$ and parity $a,$ we have for $t>0$ and $\sigma>1,$
    \begin{align}
        \frac{1}{2 \pi i} \int_{(\sigma)} \Lambda(s,g) t^{-\frac{s+a}{2}} ds &=  \frac{1}{2 \pi i} \int_{(\sigma)} \Bigl(\frac{N}{\pi}\Bigr)^{\frac{s+a}{2}} \Gamma \Bigl(\frac{s+a}{2}\Bigr) L(s,g) t^{-\frac{s+a}{2}} ds \\
        &=  \sum_{n \ge 1} g(n) \frac{1}{2 \pi i} \int_{(\sigma)} \Bigl(\frac{N}{\pi}\Bigr)^{\frac{s+a}{2}} \Gamma \Bigl(\frac{s+a}{2}\Bigr) n^{-s} t^{-\frac{s+a}{2}} ds \\ 
        &= 2 \sum_{n \ge 1} g(n) n^a e^{- \pi n^2t/N} = \psi(t,g) \label{eq: mellin inv},
    \end{align}
    where the first equality in \eqref{eq: mellin inv} follows by Mellin inversion for $e^{-x}$ and the second follows because $g$ has parity $a.$

    Now fix $c \in H$ and $\sigma>1.$ Applying \eqref{eq: mellin inv} to $f^c$, we get
    \begin{equation}
        \psi(t,f^c)= \frac{1}{2 \pi i} \int_{(\sigma)} \Lambda(s,f^c) t^{-\frac{s+a}{2}} ds, \qquad t>0.
    \end{equation}
    Using the functional equation \eqref{eq: func eq lambda fc} for $\Lambda$ and applying the change of variables $z=1-s$, we obtain 
    \begin{equation}
        \psi(t,f^c)=  \frac{\varepsilon_c}{2 \pi i} \int_{(1-\sigma)}  \Lambda(z,\overline {f^c}) t^{-\frac{1-z+a}{2}} dz, \qquad t>0.
    \end{equation}
    However, for a function $g$ as above, using periodicity of $g$, $L(s,g)$ can be written as a linear combination of Hurwitz zeta functions
    \begin{equation} \label{eq: L(s,g)}
        L(s,g) = \frac{1}{N^s} \sum_{r=1}^N g(r) \zeta \bigl(s, \frac{r}{N} \bigr).
    \end{equation}
    Therefore, $L(s,g)$ is holomorphic on $\C \setminus \{1\}$ and has at most polynomial growth away from $s=1$. Moreover, since $g$ is periodic of parity $a$ and $\zeta(-n,\alpha)= - B_{n+1}(\alpha)/(n+1)$ \cite[Thm 12.13]{apostol2013introduction} and $B_k(1-x)= (-1)^k B_k(x)$ \cite[Def., p. 264]{apostol2013introduction}, where $B_n(x)$ are the Bernoulli functions defined in \cite[p. 264]{apostol2013introduction}, it follows by using \eqref{eq: L(s,g)} and pairing the $r$-th and $(N-r)$-th terms in $L(-a -2n,g)$ that
    \begin{equation} \label{eq: vanish trivial zeros}
        L(-a -2n,g)=0 , \qquad n\ge 0.
    \end{equation}
    
    In particular, since $\overline {f^c}$ is periodic of parity $a$, the above holds for $L(s,\overline {f^c})$. Also, since $L(s,\overline {f^c})$ is holomorphic at $s=1$\footnote{We assumed that $L(s, f^c)$ is holomorphic at $s=1$ for all $c \in H,$ which implies holomorphy of $L(s,\overline {f^c})$ at $s=1$ for all $c \in H.$ This follows from the fact that  $\text{Res}_{s=1} L(s,g)= \frac{1}{N} \sum_{r=1}^N g(r)$, and it is thus clear, by conjugation, that $\text{Res}_{s=1} L(s,g)=0$ iff $\text{Res}_{s=1} L(s, \overline{g})=0.$}, it follows that $L(s,\overline {f^c})$ is entire. Therefore, using \eqref{eq: vanish trivial zeros} and the definition of $\Lambda$ in \eqref{eq: Lambda def}, we observe that $\Lambda(s,\overline {f^c})$ is entire.
    
    Therefore, using Cauchy's theorem, the holomorphy of $\Lambda$ on $\C$, the at most polynomial growth of $L(s,\overline {f^c})$, and Stirling's formula, we can shift the line of integration from $\Re(s)=1-\sigma$ to $\Re(s)=\sigma$ and then use \eqref{eq: mellin inv} with $g=\overline {f^c}$, so that
    \begin{align}
        \psi(t,f^c)&= \varepsilon_c t^{-a-\frac{1}{2}} \frac{1}{2 \pi i}  \int_{(\sigma)}  \Lambda(z,\overline {f^c}) (t^{-1})^{-\frac{z+a}{2}} dz=  \varepsilon_c t^{-a-\frac{1}{2}} \psi \bigl(t^{-1}, \overline {f^c} \bigr) , \qquad t>0.
    \end{align}
    Applying \Cref{thm: theta converse case 2}, the claim follows.
\end{proof}

\bibliographystyle{plain}
\bibliography{ref}

\end{document}